\documentclass[letterpaper, 12 pt, journal,onecolumn]{ieeeconf}
\usepackage{amsmath,amssymb ,graphicx,subfigure,xcolor,url,epstopdf}

\usepackage[numbers,comma,sort&compress]{natbib}
 
\newcommand{\rev}[1]{#1}%

\title{\bfseries Long-term regulation of prolonged epidemic outbreaks in large populations via adaptive control: a singular perturbation approach}
\author{M. Ali Al-Radhawi, Mahdiar Sadeghi, and Eduardo D. Sontag 
\thanks{The authors are with the Department of Electrical \& Computer Engineering, Northeastern University, Boston, MA. E.D. Sontag, is also with the Department of Bioengineering, Northeastern University.
Emails: \texttt{\{malirdwi,sadeghi.ma,e.sontag\}@northeastern.edu}.}}
\date{\today}
\newcommand\image{\mbox{image}}
\newtheorem{theorem}{Theorem}%
\newtheorem{proposition}{Proposition}%
\begin{document}
\abovedisplayskip=3pt
 
\belowdisplayskip=3pt
\maketitle
\thispagestyle{empty}
\begin{abstract}
\rev{In order to control highly-contagious and prolonged outbreaks,} public health authorities \rev{intervene} to institute social distancing, lock-down policies, and other Non-Pharmaceutical Interventions (NPIs). Given the high social, educational, psychological, and economic costs of NPIs, authorities tune them, alternatively tightening up or relaxing rules, with the result that, in effect, a relatively flat infection rate results.  
For example, during the summer in parts of the United States, daily \rev{COVID-19} infection numbers dropped to a plateau.
This paper approaches NPI tuning as a control-theoretic problem, starting from a simple dynamic model for social distancing based on the classical SIR epidemics model.
Using a singular-perturbation approach,
the plateau becomes a 
Quasi-Steady-State (QSS) of a reduced two-dimensional SIR model regulated by adaptive dynamic feedback.
It is shown that the QSS can be assigned and it is globally asymptotically stable. Interestingly, the dynamic model for social distancing can be interpreted as a nonlinear integral controller.  Problems of data fitting and parameter identifiability are also studied for this model.  The paper also discusses how this simple model allows for a meaningful study of the effect of population size, %
 vaccinations, and the emergence of second waves.
\end{abstract}
\section{Introduction}
 COVID-19, a highly contagious disease, has been spreading globally and has already claimed more than 2.5 million lives  during its first year. %
 Unsurprisingly, this has sparked a renewed  interest in the dynamical modeling and analysis of infectious diseases, particularly in the control theory and dynamical systems communities 
\cite{weitz20,parisini20,franco20,levine20,giordano20,leonard20,johnston20,sontag20,sadeghi20}.  
The starting point in modern epidemiological modeling   is the Kermack-McKendrick model  \cite{kermack27,hethcote00} which is known as the Susceptible-Infectious-Removed (SIR) model. It assumes a well-mixed homogeneous population, and it can be written as  %
the three-compartment model:
\begin{align}\nonumber
\dot S(t) &= %
-c\beta S(t) I(t),  \\ \label{sir}
\dot I(t) &=  \;\;\; c\beta  S(t) I(t)  - \gamma I(t), \\ \nonumber
\dot R(t) &= \qquad\qquad\qquad\;\; \gamma I(t),
\end{align}
where $S(t),I(t),R(t)$ refer to the susceptible, infective, and removed individuals at time $t$. 
The product $b=c\beta$ and the parameter $\gamma$ are called the \emph{infection rate} and the \emph{removal rate}, respectively.  We factored the infection rate as $b=c\beta$, where we call $c$ and $\beta$ the \emph{intrinsic infection rate} and the \emph{contact rate} respectively, to emphasize that $b$ depends on both \emph{biological} and \emph{societal} conditions. %

\rev{The  steady-state behavior of the above model is simple.} Infectives always converge  to zero, and the  susceptibles converge to a residual nonzero value. %
 Nevertheless, the \emph{transient} dynamics are much more interesting. Starting from any nonzero number of infectious individuals, since
$\dot I(0) = (bS(0)-\gamma)I(0)$,
an outbreak occurs iff $R_0:=b S(0)/\gamma > 1$, otherwise, $I(t)$ decreases to zero. In the case of an outbreak, the infectives increase to a peak value, and then gradually converge to zero. 

However, \rev{in prolonged outbreaks,} %
the transient behavior predicted by the basic SIR model is no longer representative of the outbreak after the initial surge. This is since the increased number of infections will necessarily produce a feedback effect via media coverage, public precautionary measures, and government interventions. Mathematically, this means that \emph{the effective contact rate} $\beta$ is no longer constant over time. Modeling and analysis of such feedback effects have been a topic of a large body of research.  For our purposes, there are two categories of models:  \emph{compartmental} models, and   \emph{compound models}. The state vector in the first class of models consists of sub-populations that sum to the total population.  In such context, one of the early strategies for modeling feedback is nonlinear contact rates \cite{capasso78,levin86,korobeinikov05,franco20,weitz20}.  For example, the following form has been studied in \cite{capasso78}:
$ \beta(t)= \frac{\beta_0}{1+k I(t)}, $
where $\beta_0$ is the nominal contact rate.  A
quasi-mechanistic justification of this term was proposed in the language of reaction networks as a two-stage process \cite{franco20}, which produces the above expression %
after a time-scale separation argument. %
Other models of feedback include regulation   by direct control of the contact rate     \cite{levine20,sontag20,sadeghi20}, by the creation of tested and quarantined compartments  \cite{giordano20,sontag20}, or by vaccinations \cite{hansen11}.

 However, such models do not usually account for the fact that the feedback effects originating from  public interventions and fear of infection are  not fast enough to justify   eliminating them from the dynamical modeling. %
  Therefore, there have been a second category of models that we call \emph{compound models}. They  describe the pandemic as an interaction between two dynamical processes, namely the \emph{compartmental component} which corresponds to the dynamics of disease, and the \emph{regulatory component} which corresponds to  dynamics of the meta-effects of the disease. The latter dynamics have been referred to by multiple names such as ``dynamics of fear'' \cite{epstein08,johnston20},  ``information transmission'' \cite{funk09}, ``behavior dynamics'' \cite{wang15}, ``reactive social distancing'' \cite{yu17}, and the population observation of the infection level \cite{leonard20}.   Despite the different names, the latter paradigm can be summarized as an interaction of two types of dynamics (Figure \ref{f.block_d}-a). In control theoretic terms, this can be interpreted as the familiar plant-controller dichotomy.
 
 \begin{figure}
	\centering
 {\includegraphics[width=0.879\columnwidth]{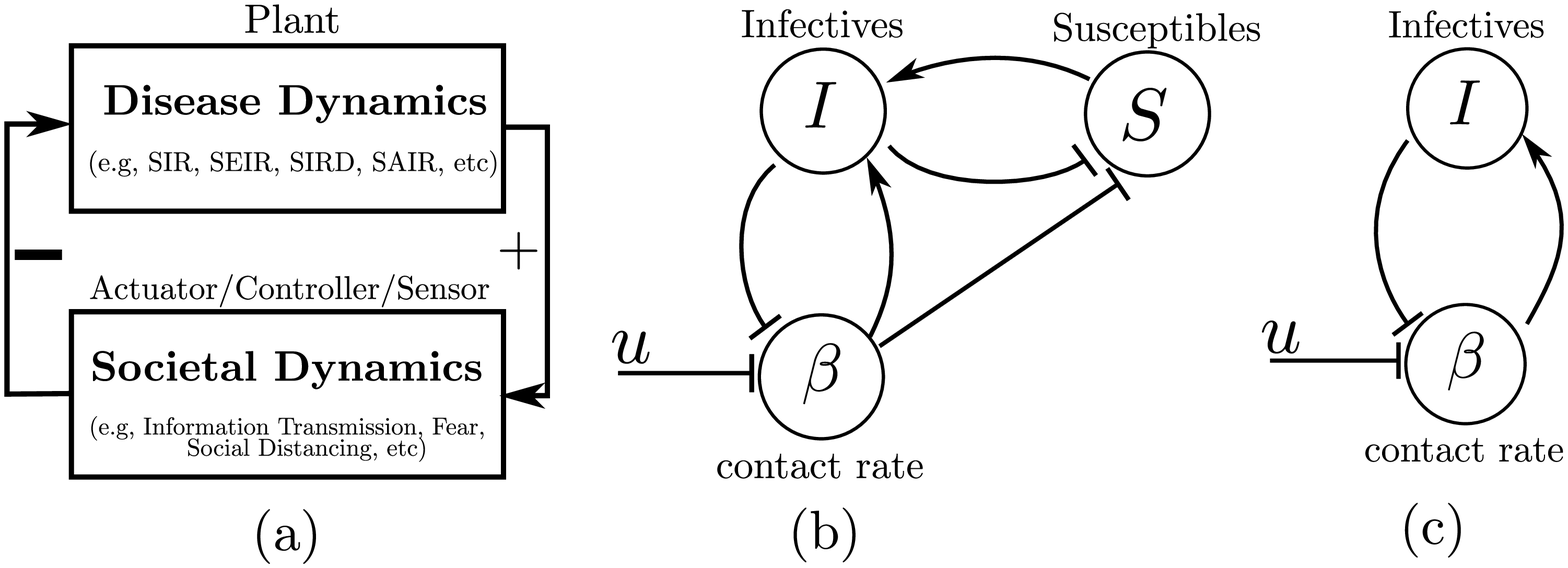}} \\
{\includegraphics[width=0.66\columnwidth]{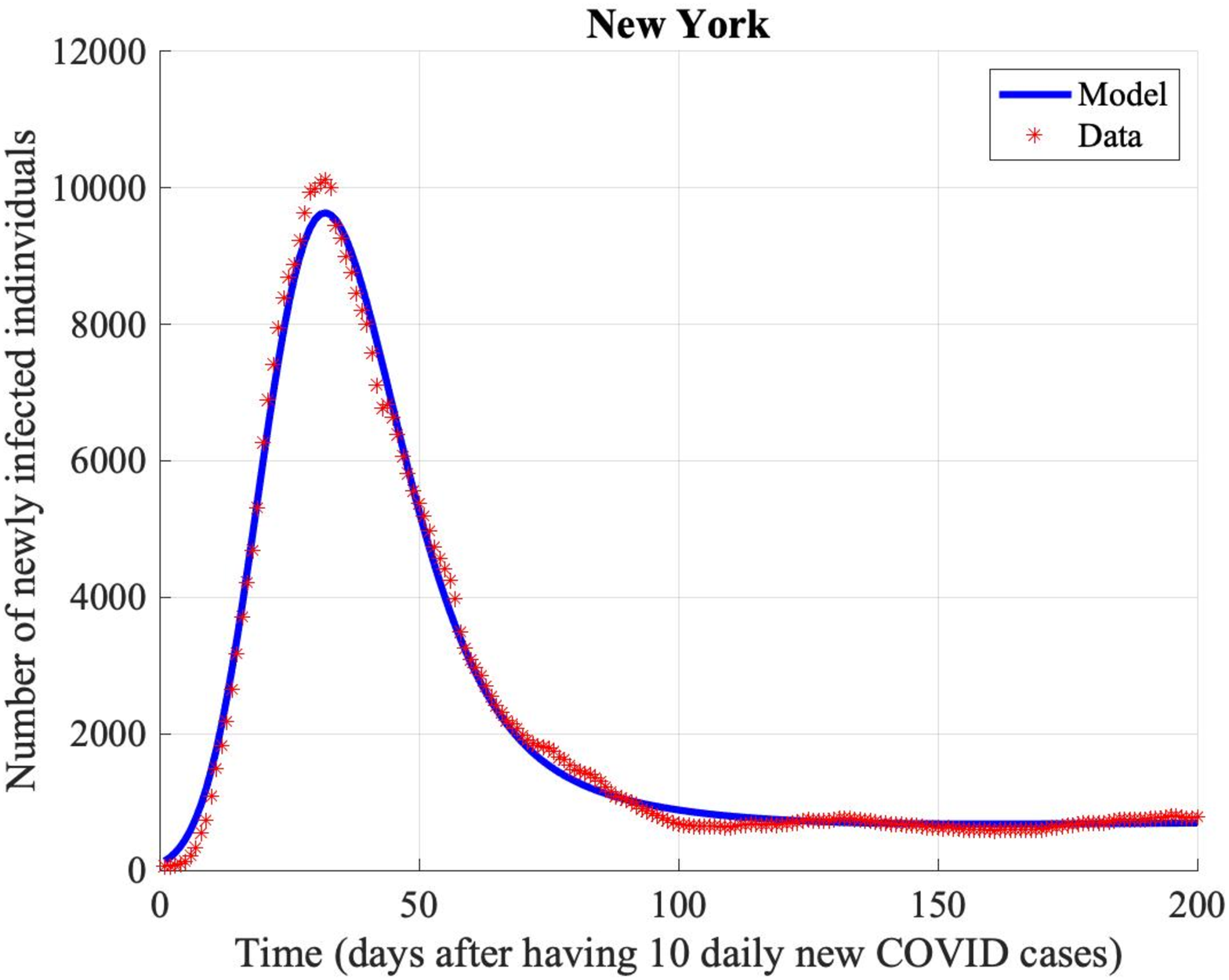}}\\ (d) \vspace{-0.12in}
	\caption{ \textbf{Modeling   \rev{a prolonged outbreak}.} (a) The compound epidemic model   can be cast in a standard control theoretic framework. (b) A minimal regulated SIR model, (c) The reduced regulated SIR model in the case of a large population over short time periods (e.g., less than a year). (d) Fitting the first 200 days of the new daily \rev{COVID-19} infections (moving average of 7 days) in the state of New York using the two-dimensional model \eqref{fit_model}. \rev{The fitted parameters are: $\gamma=0.071$, $\alpha=0.0575$, $K=0.0104$, and $u=0.8\!\times\! 10^{-4}$.}} \label{f.block_d}
\end{figure}

In this work, we  present a  control-theoretic study of a minimal model \rev{of a prolonged outbreak} that captures both of the disease dynamics and the social dynamics using tools from singular perturbation theory and nonlinear control. Unlike other elaborate models in the literature, we opt for an intelligible model (Fig. \ref{f.block_d}-b,c) that is simple enough to studied analytically, yet it is also representative of the dynamics. As a demonstration, \rev{we} consider the daily \rev{COVID-19} infection numbers in the state of New York. Using our two-dimensional model (Eq. \eqref{fit_model}) we can fit the data satisfactorily  as in Fig. \ref{f.block_d}-d. Furthermore, the model allows us to interpret the plateau in the second hundred days as a QSS of the fast component of the model and understand the role of the parameters. 

Organization is as follows. In \S 2, we propose our modeling framework. We study stability and the validity of the time-scale decomposition in \S 3, and the properties of the model including integral error tracking and fold-change detection in \S 4. In \S 5, we study the specific case of Monod-inhibition and parameter identifiability, and \S 6 contains fits and a discussion.

\section{Regulating the SIR via Adaptive Control}
\subsection{Regulation of the contact rate}
Consider the SIR model \eqref{sir}. In order to model the time-varying nature of $\beta$ in a prolonged epidemic, we start from a one-dimensional ODE regulating its dynamics. In the absence of disease, let $\beta_e$ be the nominal value of the contact rate and we assume that it is asymptotically stable. First, we write a simple linear model (which can be interpreted as a linearization of a more general nonlinear model which will \rev{be considered} later) as:
$\dot \beta = - \alpha ( \beta - \beta_e),$
where $\alpha>0$ is the rate at which the contact rate returns to the nominal rate after the absence of an external perturbation. 

After the emergence of an epidemic, we model the feedback effect as follows: the society aims at reducing the nominal contact rate to a new reduced level that varies in a direction opposite to that of %
$I$.  Let the new nominal contact rate be $\beta_e = h(I)$ where $h$ is a $C^1$ \rev{ positive function. It quantifies the new desired contact rate when the number of infectives rises, and hence $h$ is naturally strictly decreasing.} Furthermore, we allow a control ``knob'' $u>0$ to tune   the \emph{society's perception   of the severity of the disease}. Using common parlance, we can say that $u$ is \rev{a parameter to modulate the intensity of  \emph{social distancing}}.   Thus, we arrive to the model:
$ \dot \beta = - \alpha ( \beta - h(uI)),$ where $u>0$ is a constant.  The theory can be expanded effortlessly if we consider  a slightly more general model:
\begin{equation}\label{beta}
	 \dot \beta = - \alpha ( g(\beta) - h(uI)),
\end{equation}
where $g$ is any strictly increasing nonnegative $C^1$ function. \rev{It models the natural tendency of people to socially distance, and hence it is proportional to the current contact rate $\beta$.} For instance, consider a city-wide festival. After it ends and the external stimulus disappears, \rev{   people intrinsically want to reduce their current contact rate  and return to their nominal one}. We will see later that this natural social distancing tendency acts as a damping control that stabilizes the system.  %

\subsection{Model reduction via time-scale separation}
Large cities and metropolitan areas are fertile grounds for the spread \rev{of highly-contagious diseases}. Due to the disease suppression measures, the number of susceptibles remains very large compared to the infectives (e.g., millions vs thousands). Hence, the susceptibles can be considered constant from the point of view of the infectives in an appropriate time scale of interest (e.g., 6 months).  To be more precise, we use singular perturbation techniques \cite{kokotovic99} to write \eqref{sir},\eqref{beta} as a slow-fast decomposition. Let $0<\varepsilon \ll 1$ be a small parameter. We write the re-scaled susceptibles $\tilde S=\varepsilon   S$. For instance, if the initial number of susceptibles is $20 \times 10^6$ (e.g., the metropolitan New York area) and we let $\varepsilon=10^{-6}$, then $\tilde S$ has the unit of Millions. Furthermore, we let $\tilde c = c/\varepsilon$. This re-scaling is meaningful because we have observed in our data fitting that $c$ is in the order of $\varepsilon$. Note that $cS=\tilde c \tilde S$. Therefore, we can write \eqref{sir}, \eqref{beta} as:
\begin{align}\label{slowt}
	\dot{\tilde S}&=  - \varepsilon c \beta SI = -\varepsilon \tilde c \beta \tilde S I, \\ \label{fast1}
	\dot I & = \tilde c \beta \tilde S I - \gamma I, \\ \label{fast2}
	\dot \beta & = - \alpha( g(\beta) - h(uI)). 
\end{align}
The dots indicate derivative with respect to time $t$ in the \emph{original time scale}, which corresponds to the \emph{fast} time-scale in the singular perturbations literature. 
We refer to \eqref{fast1}-\eqref{fast2} as the fast subsystem, and to \eqref{slowt} as the slow system.%

\section{Stability and Singular Perturbation Analysis} 
\subsection{Steady-state analysis of the fast system}
For the fast system \eqref{fast1}-\eqref{fast2}, the variable corresponding to the rescaled susceptibles is seen as a constant. Hence, for simplicity of notation denote $c_s :=   c   S$, and we study the following system in the \emph{original} time scale:
\begin{align}\label{2dmodel}
	\dot I &= (c_s \beta - \gamma ) I, \\ \nonumber
	\dot \beta & = -\alpha ( g(\beta)-h(uI) ),
\end{align}
where $u>0$ is fixed. We recall our assumptions: \\
\textbf{(A1)} $h,g$ are $C^1$ functions. \textbf{(A2)} $h$ is strictly decreasing and positive on $[0,\infty)$. \textbf{(A3)} $g$ is strictly increasing and nonnegative on $[0,\infty)$. 

It follows from the assumptions above that \eqref{2dmodel} is positive, i.e., the orthant is forward invariant.
By A1-3, the system \eqref{2dmodel} can admit (up to) two steady states $(I_d,\beta_d), (I_e,\beta_e)$ (which are QSSs considering the full model): 
\begin{enumerate} \item The \emph{disease-free} steady state which has $I_d=0$%
	, and it exists iff $h(0) \in \image(g)$.  

\item The \emph{endemic steady state} which has $I_e>0$.   %
 It exists iff $g(\gamma/c_s) \in \image (h)^{\circ}$, where $^\circ$ denotes the interior of a set. %
\end{enumerate}

For a given contact rate model \eqref{beta}, the endemic steady state    can fail to exist only if $\gamma$ is sufficiently large, i.e., if the infectives are removed quickly, or if $c_s$ is sufficiently small. So, either the intrinsic infection rate is small (low infectivity), or $\tilde S$ is small (small number of susceptibles).%

We start by this basic result:
\begin{proposition} Consider \eqref{2dmodel} with Assumptions A1-3. If   the endemic   steady state   exists,   then the disease-free steady state is exponentially unstable. %
	\end{proposition}
\begin{proof}
	If the disease-free steady state does not exist then the statement is vacuously true.  Hence, assume that $h(0) \in \image(g)$. Note that $\beta_d=g^{-1}(h(0))$.
	Writing the Jacobian of \eqref{2dmodel} at $(I_d,\beta_d)$, we get:
	\[ \begin{bmatrix}  c_s g^{-1}(h(0)) - \gamma & 0 \\  \alpha u \left .\frac{\partial h}{\partial I}\right |_{I=0}  & -\alpha \left . \frac{\partial g}{\partial \beta} \right |_{\beta=\beta_d}
		\end{bmatrix}.\]
	Since the Jacobian is upper triangular,  the eigenvalues coincide with the diagonal entries. Since the endemic steady state exists, then $g(\gamma/c_s) \in \image (h)^\circ$. For the sake of contradiction, assume that $g^{-1}(h(0)) < {\gamma}/c_s$. Since $g$ is increasing then $h(0) < g(\gamma/c_s) $. Since $h$ is decreasing, then $ 0 > h^{-1}(g(\gamma/c_s))$ which is a contradiction since the endemic steady state is positive. Therefore, the Jacobian has a positive eigenvalue and hence the disease-free steady state is exponentially unstable.
	\end{proof}
 
 Since we are interested in the long-term regulation of the epidemic, we assume that the endemic state   exists; i.e., we have:\\
 \textbf{\rev{(A4)}}   $g(\gamma/c_s) \in \image (h)^\circ$.
 
 \subsection{Stability analysis of the fast system}
 We examine now the stability of  the endemic steady state. We perform the following change of coordinate $p:= \ln I$. Hence, we get the following system:
 \begin{align}\label{2dmodel_t}
 	\dot p &= c_s \beta - \gamma, \\ \nonumber
 	\dot\beta &= -\alpha ( g(\beta)- h(ue^p)).
 	\end{align}
 
 In order to analyze the system, we can momentarily assume that $g\equiv 0$. In this case, the system is a Hamiltonian system. When $g$ is added, which represents the intrinsic tendency of people to socially distance, then it can be interpreted as a \emph{nonlinear damping control} that stabilizes the system, and the Hamiltonian can be used as a Lyapunov function \cite{mct_sontag,sontag11}.
 
 This is stated in the following theorem:
 \begin{theorem}\label{th.gas}
 	Consider \eqref{2dmodel} with Assumptions A1-A4. Then, the endemic steady state $(h^{-1}(g(\gamma/c_s))/u,\gamma/c_s)$ is a unique positive steady state which is  globally asymptotically stable and locally exponentially stable.
 	\end{theorem}
 \begin{proof} Let us consider the transformed system \eqref{2dmodel_t}. Recall $\beta_e=\gamma/c_s$. By A4, there exists $p_e$ such that $g(\gamma/c_s)=h(ue^{p_e})$. Then, consider the Lyapunov function:
 	\[V(p,\beta)=\int_{p_e}^p \left ( h(u e^{p_e}) - h(u e^{\tilde p}) \right ) \, d\tilde p + \frac 12 (c_s \beta - \gamma)^2. \]
 	We first show that $V$ is positive definite. To that end, note that $V(p,\beta)\ge 0$ follows since the decreasingness of $h$ implies that $h(u e^p_e)\ge h(ue^p)$ for all $p\ge p_e$. Furthermore, $V$ vanishes only at the steady state $(p_e,\beta_e)$. 
 	
 Noting that $h(ue^p)-g(\beta) = h(ue^p)- h(u e^{p_e}) + g(\gamma/c_s) -g(\beta)  $, we write:
\begin{align*}\dot V(I,\beta)&= \left ( h(u e^{p_e}) - h(u e^{\tilde p}) \right )(c_s \beta - \gamma) + (c_s \beta - \gamma) (h(ue^p)- h(u e^{p_e}) + g(\gamma/c_s) -g(\beta)) \\
&= (c_s \beta - \gamma ) (g(\gamma/c_s) -g(\beta)) \le 0,\end{align*}
where the last inequality follows since $g$ is increasing. To infer global stability, we use  LaSalle's invariance principle \cite{khalil} which requires  the set of steady states to be the only invariant set contained in the kernel of $\dot V$. The calculation above   shows that all trajectories in $\mbox{kernel} (\dot V)$ have $\beta=\gamma/c_s$. Hence, consider an invariant set formed of a trajectory with $\beta(t)\equiv \gamma / c_s \in \mbox{kernel} (\dot V)$. This implies that $\dot \beta (t) \equiv 0$. Hence, $g(\beta (t))=g(\gamma/c_s)\equiv h(ue^{p(t)})$. By A1-A4, we have $p(t)\equiv p_e$. Hence, LaSalle's condition is satisfied and global stability follows for \eqref{2dmodel_t}. Hence, the endemic steady state of the system \eqref{2dmodel} is globally asymptotically stable.

To show exponential stability, we write the characteristic equation of the Jacobian of \eqref{2dmodel} at the endemic steady state as:
 $\lambda^2 + \alpha \left . \frac{\partial g}{\partial \beta} \right |_{\beta=\beta_e} \lambda - c_s I_e \alpha u \left .\frac{\partial h}{\partial I}\right |_{I=I_e}=0$ which has strictly positive coefficients (by A1-A4). Hence, the Routh-Hurwitz criteria implies that the eigenvalues have strictly negative real parts. Hence, local exponential stability follows.
 	\end{proof}
 
 \subsection{Validity of the slow-fast decomposition}
 The system \eqref{slowt}-\eqref{fast2} is written in the fast time-scale which allows analysis of the fast dynamics while treating the slow dynamics as constant.
 In order to perform singular-perturbation analysis, we formulate the system \eqref{slowt}-\eqref{fast2} in the standard form. To that end, we write the system in the \emph{slow time scale} $\tau=\varepsilon t$.  Hence, we get:
 \begin{align} \nonumber
 	d{\tilde S}/d\tau&=    (\tilde c \tilde S) \beta I, \\  \label{fastT}
 	\varepsilon 	d  I/d\tau & = (\tilde c  \tilde S) \beta I - \gamma I, \\ \nonumber  
 	\varepsilon 	d\beta/d\tau & = - \alpha( g(\beta) - h(uI)). 
 \end{align}

In order to verify that the slow-fast approximation is valid, we use Tikhonov's Theorem \cite{kokotovic99,khalil}. 
The stability conditions needed  are guaranteed by Theorem \ref{th.gas}. 

In the slow time scale,  the fast subsystem equilibrates to its quasi-steady-state which is a function of the susceptibles. Hence, the variables $I,\beta$ can be approximated by   $ I_e=h^{-1}(g(\gamma/(cS)))/u,  \beta_e =\gamma/(cS)$.  Then, the \emph{slow system} can be written as follows:
\begin{equation}\label{fast} d \bar S/d\tau  = - \frac{\gamma}{u} h^{-1}(g(\gamma/(c\bar S))),
	\end{equation}
whenever A4 is satisfied, and  $d\bar S/d\tau=0$ otherwise.

We state the following result:
\begin{proposition}(Tikhonov's Theorem) Consider the system \eqref{fastT} defined on $[0,T]$ for some $T>0$, and let $S(\tau,\varepsilon), I(\tau,\varepsilon), \beta (\tau,\varepsilon)$ be its solutions for initial conditions $S_0,I_0,\beta_0$. Assume A1-A4 hold. Then, there exists $\varepsilon^*>0$ such that for all $I_0, \beta_0$ and all $0<\varepsilon<\varepsilon^*$, we have: $\tilde S(\tau,\varepsilon)- \bar S(\tau) = O(\varepsilon), I(\tau,\varepsilon)- \bar I(\tau/\varepsilon) = O(\varepsilon), \beta(\tau,\varepsilon)- \bar  \beta(\tau/\varepsilon) = O(\varepsilon) $ hold uniformly for $\tau \in[0,T]$, where $\bar S(.)$ is the solution of \eqref{fast}, and $\bar I(.),\bar \beta(.)$ are the solutions of $\eqref{2dmodel}$ over the original time scale. 
	\end{proposition}
	\begin{proof}
	    The result follow from \cite[Theorem 11.1]{khalil} and Theorem \ref{th.gas}.
	\end{proof}

 \section{The Regulated SIR as an Integral Control System}
 \subsection{Properties of the regulated SIR}
 \subsubsection{The endemic steady-state contact rate is independent of   the societal dynamics} \label{s.beta_independent} If the endemic steady state exists, then the steady state level of the contact rate is $\gamma/c_s$, which is independent of the actual dynamics of $\beta$. In other words, $\beta_e$ is robust to changes in the modeling of the contact rate. For any choice of $g,h,u,\alpha$, if the infection is  endemic, then \emph{the (quasi) steady state level of the contact rate is determined solely by the parameters of the epidemic}.
 \subsubsection{The infection as a Proportional-Integral (PI) controller} The system in \eqref{2dmodel} is usually viewed as the contact rate regulating the dynamics of the infection. Conversely, we can view it as \emph{the infection regulating the dynamics of the contact rate}. Viewing the system from this angle, we note that the infective compartment \eqref{2dmodel} is a PI  controller tracking  the error between the growth rate of infectives $c_s\beta$ and the removal rate $\gamma$ with a state-dependent gain $I$. (In the transformed coordinates \eqref{2dmodel_t}, $p$ is a linear PI controller). In other words, the system is trying to stabilize $\beta$ to a level ensuring that the infection stays at a QSS, i.e., $\dot I=0$. 
 
 \subsubsection{Adaptation to step inputs} As can be noted above, the (quasi)  steady level of the contact rate is independent of $u$ and is regulated by a PI controller. Therefore, if there is a step input applied via $u$ (meaning that the society's perception  to the severity of the epidemic changes to a new level), then the contact rate rejects this change and returns to its original level over the long term.  
 Such regulatory modules are ubiquitous in biological networks \cite{sontag11}.
  \subsection{Fold Change Detection.} In addition to adaptation, an important property featured by many adaptive biological networks is their ability to ignore changes to the absolute level of any input, and  only detect fold changes to the input \cite{shoval10,sontag11}. More precisely,  assume that the system \eqref{2dmodel} is at a steady state \rev{$(I_e,\beta_e)$}  \rev{for a constant input} $u(t)=\bar u$ at time $t=0^-$. Assume that the input changes to $u(t)=q \bar u$ for all $t>0$ and for some $q>0$. Consider the following parameterized output $\beta(t;q,\bar u)$ starting from the initial condition $(I_e,\beta_e)$. Then, the adaptive system \eqref{2dmodel} is said to have a Fold-Change Detection (FCD) property if the output  $\beta(t;q,\bar u_1)=\beta(t;q,\bar u_2)$ for all $t>0$ and any $\bar u_1,\bar u_2>0$. In other words, the output trajectory depends on the fold-change $q$, and not on the absolute level of the input. \rev{We prove next that this applies to our model (see \cite[Lemma 3.1]{sontag11} for a more general result):}
  \begin{theorem}
  Consider \eqref{2dmodel} with Assumptions A1-A4, and let the contact rate $\beta$ be the output. Then, the input-output system has an FCD property.
  \end{theorem}
\begin{proof} \rev{Given $\bar u_1, \bar u_2>0$, $q>0$. Observe that $I(0;q,\bar u_1) =\tfrac{\bar u_1}{\bar u_2} I(0;q,\bar u_2)$, and that \eqref{2dmodel} is invariant with respect to transformation $I \mapsto  \tfrac{\bar u_1}{\bar u_2} I$. Hence, $I(t;q,\bar u_1) =\tfrac{\bar u_1}{\bar u_2} I(t;q,\bar u_2)$ for all $t$.  Write $\dot\beta=F(\beta,uI)$. Then, $\dot \beta(t;q,\bar u_1)= F(\beta,q\bar u_1 I(t;q,\bar u_1))=F(\beta,q\bar u_2 I(t;q,\bar u_2))=\dot \beta(t;q,\bar u_2)$. Since $\beta(0;q,\bar u_1) =\beta(0;q,\bar u_2)$, FCD follows.   %
	}
	\end{proof}

This result implies that the regulated SIR model does not care about the absolute level of the perception of infection $u$, instead it only cares about the fold changes in its level. For instance, a change from a base level of $u=1$ to $10$, will cause the same transient response as a change from $u=10$ to $u=100$. This is consistent with the common understanding of how human perception of fear and danger works.
 
 \subsection{Assigning the quasi-steady-state level of the infectives.}   If ending the infection is not feasible in the short term (which corresponds to having an endemic (quasi) steady state), then the usual aim of an effective contact rate regulation policy (more commonly known as \emph{a social distancing policy}) is to keep the number of infectives plateaued at a low small number. We show in this subsection, that the $I$-coordinate of the endemic steady state can be assigned to any desired level via the constant input $u$.
 
 \begin{proposition}
 	Consider the system \eqref{2dmodel} with  A1-A4. Let $\rev{I^*}>0$ be a desired quasi-steady-state level of the infectives. Then, the required control input is $u= h^{-1}(g(\gamma/c_s))/\rev{I^*}$.
 \end{proposition}
\begin{proof} The expression above is well-defined by A1-A4. Setting $\dot \beta=0$ in \eqref{2dmodel} and substituting $ \beta_e=\gamma/c_s$ yields the required expression.
	\end{proof}

 \section{Monod-type Inhibition and Identifiability}
 \subsection{Monod-type Inhibition}
 \begin{figure}
 	\centering \includegraphics[width=0.66\columnwidth]{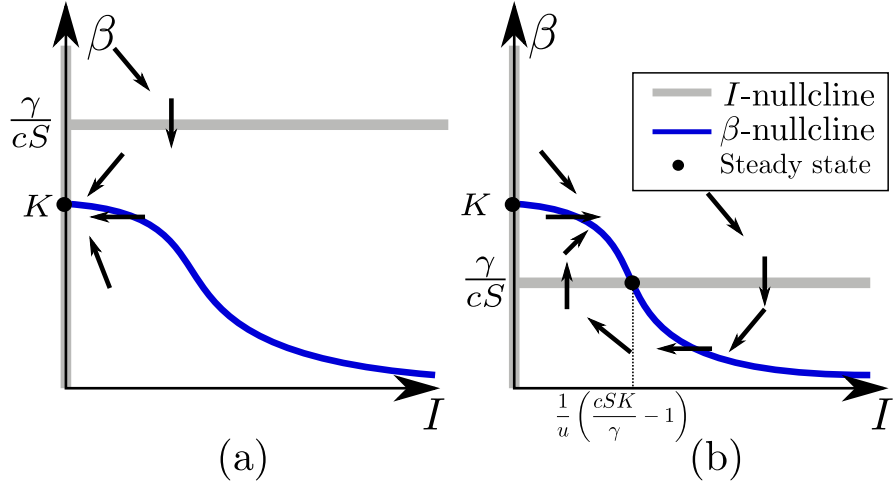}
 	\caption{ \textbf{Phase plane analysis of the fast subsystem with Monod inhibition.} (a) The disease-free steady is asymptotically stable when $\frac{\gamma}{cS}>{K}$.  (b) The endemic steady state is born, and the disease-free steady state exchanges stability with the endemic one. } \label{phaseplane}
 \end{figure}
  In order to examine the proposed model more concretely, we will use  $h: x \mapsto K/(1+x), g:z \mapsto z$, which is a popular functional form in the epidemics literature \cite{capasso78}, and it is known as the Monod or the Michaelis-Menten form. Hence, the contact rate equation becomes:
 \begin{equation}\label{monod} \dot \beta = -\alpha \left (\beta-\frac{K}{1+uI}\right ),\end{equation}
  where $K$ can be interpreted as the nominal contact rate. 
 
 Since the fast system is two-dimensional, we plot the phase plane  in Fig. \ref{phaseplane}. Condition A4 (guaranteeing the existence of endemic steady state) translates into $S >S^*=\gamma/(cK)$. %
 The minimum number of susceptibles needed for a sustained endemic is proportional to the removal rate $\gamma$, and inversely proportional to the intrinsic infection rate $c$ and the nominal contact rate $K$. In Fig. \ref{phaseplane}-a), the disease-free steady state is asymptotically stable. However, when the endemic steady state is born, it becomes the only asymptotically stable steady-state. 
 
{   
 The approximated equation for the susceptibles over the slow time-scale is given by \eqref{fast}. It can be shown easily that if $\tilde S(0)>S^*$, then A4 is satisfied for all $\tau>0$. Hence, the evolution of $\tilde S$ can be describe by linear first order equation:
 $ d{\tilde S}/d\tau = \tfrac{\gamma}{u} - \tfrac{\tilde cK}{ u} \tilde S$.
 The closed-form solution in the original time scale can be written as: $\tilde S(t)= S^* + \left( \tilde S(0) - S^* \right) e^{-\tilde cK \varepsilon t / u}$. By substitution in the quasi-steady-state expression of $I(t)$ we get: $I(t)=\frac 1u \left   ( \tilde S(0)- S^* \right) e^{-\tilde c K \varepsilon t/u}$. The steady state behavior of the model is  $S(t) \to S^*$ and $I(t) \to 0$. Hence, the number of susceptibles approaches the minimal number $S^*$, and the pandemic asymptotically approaches  a disease-free steady state. Furthermore, as predicted by intuition, a higher input $u$ (i.e., stricter social distancing) means that the pandemic gets more prolonged while having a lower quasi-steady-state level of the infectives as given in Proposition 3.

}

  \subsection{Data fitting and identifiability}
In order to apply the proposed model to a practical setting, we need to deal with two issues to ensure the well-posedness of the data fitting problem.

\subsubsection{Measured output} The published data during \rev{pandemics} consist of the daily new cases and daily deaths \rev{(e.g, COVID-19 data \cite{jhu})}. However, none of these data correspond to the state variables in our model. The number of infectives $I$ can be construed as the ``active infections'', however, there are not usually  reliable estimates of active infections, nor recoveries. Instead, notice that the beginning of every active infection is reported within the new daily numbers. Hence, we interpret the new daily numbers as the \emph{inflow} to the \emph{infectives compartment} $I$. %
Therefore, the measured output $y$ can be written as: $y(t)=\tilde c \tilde S(t) \beta(t)I(t)$. As before, we assume that $\tilde S$ is approximately constant in the original time scale, and we can write: $y(t)= c_s   \beta(t)I(t)$  

\subsubsection{Identifiability} The contact rate $\beta(t)$ cannot be measured directly, and is inferred from the data. In fact, in our problem, it can be seen that a re-scaled contact rate $ \hat\beta(t)=\tilde c \beta(t)$ will produce the same measured output, and that the parameters $\tilde c$ and $K$ are not individually identifiable. Therefore, we can normalize the contact rate by fixing $\tilde c=1$. Hence, we consider the following model:
\begin{align} \nonumber \dot I(t)&= (\rev{\tilde S} \hat \beta(t) - \gamma) I(t),  \dot {\hat\beta}(t)=  \alpha\left(  \hat\beta(t)-\frac{K}{1+uI(t)}\right) ,  \\  \label{fit_model} y(t)&=\rev{\tilde S}\hat \beta(t) I(t),
	\end{align}
\rev{where $\tilde S$ is the scaled total population size} which is assumed to be constant, and unknown parameters $\gamma, K, u, \alpha$. 

Identifiability has long been studied as a special case of nonlinear observability  \cite{hermann77,sontag91}. The technique is based on computing the Lie derivatives at time zero and evaluating the rank of the observability matrix symbolically. Several toolboxes are publicly available.  The package STRIKE-GOLDD~\cite{villaverde16} finds that the problem is well-posed and both states are observable and all parameters are locally structurally identifiable. The package SIAN~\cite{sian} reports that initial states as well as all parameters are globally identifiable.
 
 \section{Simulations and Discussion}

 \subsection{The impact of population size}

\begin{figure}
	\centering
	\includegraphics[width=0.67\columnwidth]{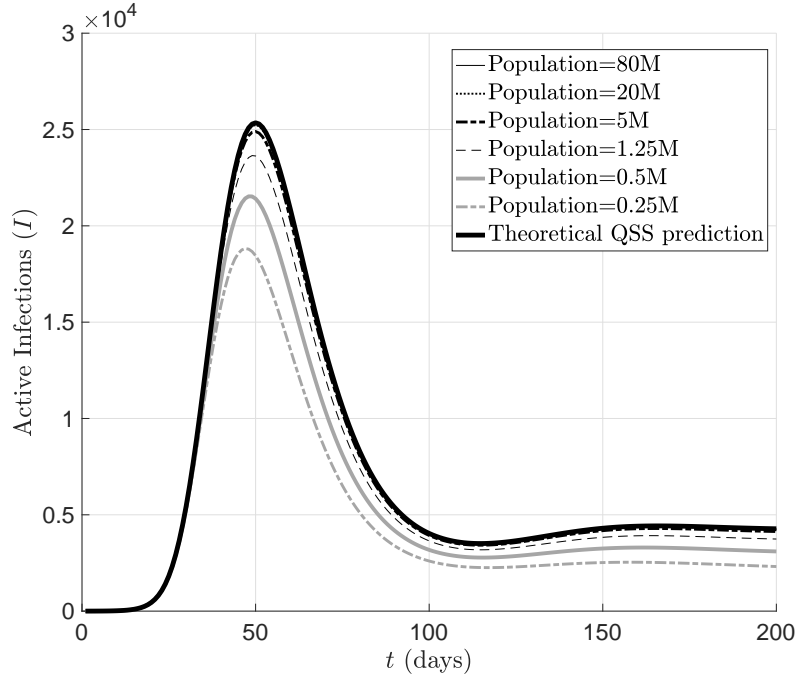} \vspace{-0.15in}
	\caption{\textbf{The effect of population size on the predicted quasi-steady-state behavior.} The total population denotes $S(0)$. For all the curves, the product $cS(0)$ is kept constant at 17.5392. The other parameters in the model \eqref{2dmodel} are: $\gamma=0.091, \alpha=0.0679, K= 0.0229$, $u=0.0008$. The QSS prediction is a solution of the reduced model \eqref{2dmodel}, while the all the other  curves are solutions of the full model \eqref{sir},\eqref{monod}. }\label{f.different_ep}
\end{figure}
We  study the sensitivity of our techniques to population size. Figure \ref{f.different_ep} shows a numerical exploration. Using our study of the reduced model, we predict that $I(t)$ converges to the  QSS level $\frac 1u\left ( \frac{cSK}{\gamma}-1\right )=$ 4,271. We \rev{simulate} the full model \eqref{sir} with different population sizes, which correspond to different values of $\varepsilon$. \rev{We keep $cS(0)$ constant for a meaningful comparison}. We start with the nominal case of a country of medium size with a population of 80M. The simulated parameters give an initial peak of active infection reaching 25,000 cases. The quasi-steady-state level is at 4,265 after 200 days which is 99.86\% close to the theoretical value that was obtained via the QSS approximation. With smaller population sizes, the theoretical value stays pretty close to the actual value for populations higher than 1M. For smaller populations, the actual behavior still resembles the theoretical behavior, however, the quasi-steady-value is lower. This not surprising since a susceptible population of 0.5M will dwindle quickly with  daily active infections in tens of thousands. 
 \subsection{Fitting to published data}
\begin{figure}
	\centering
	
	\includegraphics[width=1\columnwidth]{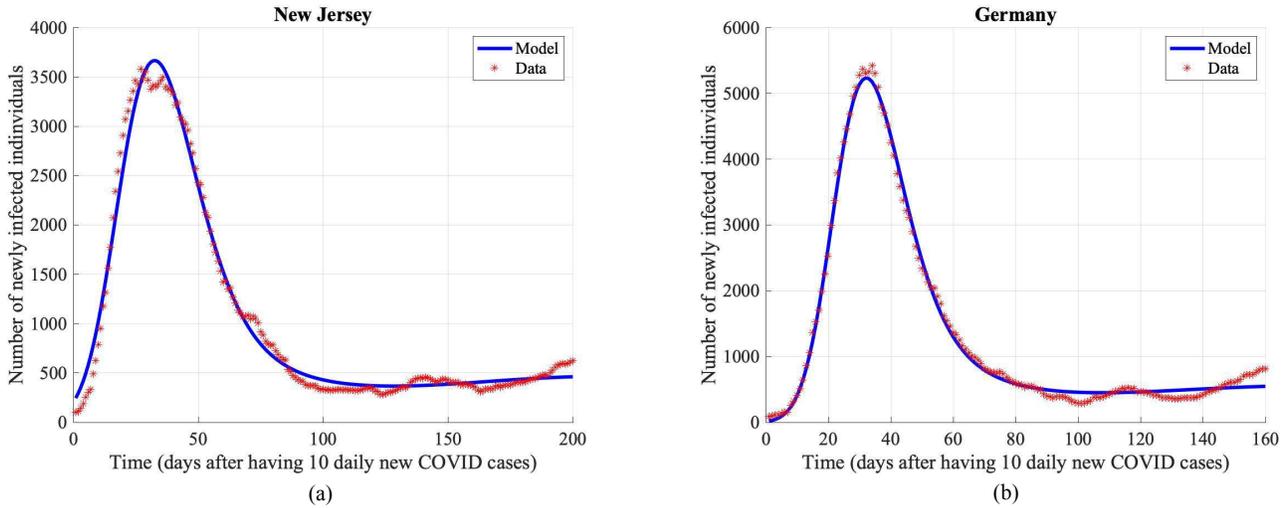}
	\vspace{-0.15in}
	\caption{\textbf{Fitting the first days of the new daily infections (moving average of 7 days).}  (a) \rev{New Jersey}  with parameters in the model \eqref{fit_model}: \rev{$ \gamma=0.071, \alpha=0.0453, K=0.0154$, $u=1.4 \times 10^{-4}$.} (b)  Germany with parameters in the model \eqref{fit_model}: \rev{$ \gamma=0.1073, \alpha=0.0613, K=0.0061$, $u=3.4 \times 10^{-4}$.}}\label{f.fits}
\end{figure}

     In order to show the effectiveness of our model, we fit the \rev{COVID-19} data~\cite{jhu} to the course of disease in  multiple countries that experienced an initial surge followed by a plateau during the summer, which we interpret as a QSS in our model. Fig. \ref{f.block_d} shows the case of the State of New York, while fits for Germany and the state of New Jersey are shown in Fig.~\ref{f.fits}. %
     Note that new daily infections data is noisy and sometime biased to weekend and weekdays. Hence, we use the standard 7-day moving average to filter the noise. %
     The numerical analysis and fitting programs are available on \url{github.com/mahdiarsadeghi/epidemicqss}.

 \subsection{Extensions: Second waves and Vaccinations}
 The proposed model has proved to be a useful tool for studying prolonged epidemics in large populations. This opens the door for studying other effects such as the emergence of second waves and the effectiveness of vaccinations. 
 
 Second waves can emerge because of desensitization to the severity of the pandemic, which can be modeled by a decrease in the control input $u$. An increase in $c$  by more infectious strains of the virus can also lead to second waves.  In addition, increased infections because of poor  ventilation and increased indoor gatherings during the winter   can be modeled by increasing $K$ or decreasing $\alpha$.  Another regulation measure of the pandemic is the use of vaccines. This can be modeled by adding the an outflow term $-\delta S$ to the susceptibles compartment.  Due space limitations, we leave  these investigations for future work.

\end{document}